\documentclass[10pt]{amsart}

\usepackage{amssymb, amsmath}

\usepackage{graphics}

\usepackage{latexsym}

\usepackage{amsmath}

\usepackage{amssymb,amsthm,amsfonts}

\usepackage{amscd}

\usepackage{tikz-cd}

\usepackage[arrow, matrix, curve]{xy}

\usepackage{syntonly}

\usepackage{stmaryrd}

\usepackage{mathtools} 

\usepackage{cases}

\usepackage{yfonts}

\usepackage{stmaryrd}

\usepackage{MnSymbol}

\usepackage{mathrsfs} 

\DeclareMathOperator{\ord}{ord}

\DeclareMathOperator{\gl}{GL}

\DeclareMathOperator{\mt}{MT}

\DeclareMathOperator{\jac}{Jac}

\DeclareMathOperator{\im}{Im}

\DeclareMathOperator{\mon}{Mon}

\DeclareMathOperator{\ima}{Im}

\DeclareMathOperator{\irr}{Irr}

\DeclareMathOperator{\PSU}{PSU}

\DeclareMathOperator{\gsp}{Gsp}

\DeclareMathOperator{\mult}{mult}

\DeclareMathOperator{\Sym}{Sym}

\DeclareMathOperator{\can}{can}

\DeclareMathOperator{\dR}{dR}

\DeclareMathOperator{\cris}{cris}

\DeclareMathOperator{\Fil}{Fil}

\theoremstyle{plain}

\newtheorem{thm}{Theorem}[section]

\newtheorem{theorem}[thm]{Theorem}

\newtheorem{lemma}[thm]{Lemma}

\newtheorem{proposition}[thm]{Proposition}

\theoremstyle{definition}

\newtheorem{remark}[thm]{Remark}

\newtheorem{definition}[thm]{Definition}

\numberwithin{equation}{thm}

\newcommand{\sC}{{\mathcal C}}

\newcommand{\sK}{{\mathcal K}}

\newcommand{\sL}{{\mathcal L}}

\newcommand{\sP}{{\mathcal P}}

\newcommand{\A}{{\mathbb A}}

\newcommand{\C}{{\mathbb C}}

\newcommand{\D}{{\mathbb D}}

\newcommand{\E}{{\mathbb E}}

\newcommand{\F}{{\mathbb F}}

\renewcommand{\H}{{\mathbb H}}

\renewcommand{\L}{{\mathbb L}}

\renewcommand{\P}{{\mathbb P}}

\newcommand{\Q}{{\mathbb Q}}

\newcommand{\R}{{\mathbb R}}

\newcommand{\Z}{{\mathbb Z}}

\newcommand{\Hom}{{\rm Hom}}

\newcommand{\Sp}{{\rm Sp}}

\begin{document}

\title{Families of Prym varieties of abelian coverings and Shimura varieties}

\author[A. Mohajer]{Abolfazl Mohajer}

\email{abmohajer83@gmail.com}



\maketitle

\begin{abstract}
Under the condition that the Prym map is injective in characteristic $p$, we prove that 
the special subvarieties in the moduli space of abelian varieties of dimension $l$ and polarization type $D$, $A_{l,D}$, arising from families of abelian covers of $\P^1$ are of a very restrictive nature. In other words, if the family is one-dimensional or if it contains an eigenspace of certain type for the group action on the cohomology of fibers, then the Shimura varieties arising from such families can only be constructed by the group action of the family. 
\end{abstract}

\section{introduction} \label{introduction}
This paper continues the studies of several authors in \cite{CF2}, \cite{CFGP}, \cite{F}. and \cite{FG} concerning families of Prym varieties and the resulting moduli varieties in $A_{l,D}$, the moduli space of complex abelian varieties of dimension $l$ with polarization of type $D$. The present work is specially motivated by Frediani's paper \cite{F}. In particular, we prove that such a moduli variety is only very rarely a totally geodesic subvariety with respect to the Siegel metric.

Let $H$ be a finite group with $n=|H|$. Suppose $C$ is a compact Riemann surface of genus $g$. Let $t:=\{t_1,\dots, t_r\}$ be an $s$-tuple of distinct points in $C$. Set $U_t:=C\setminus \{t_1,\dots, t_r\}$. The fundamental group $\pi_1(U_t, t_0)$ has a presentation $\langle \alpha_1,\beta_1,\dots,\alpha_g,\beta_g, \gamma_1,\dots,\gamma_r\mid \prod_1^s\gamma_i \prod_1^g[\alpha_j,\beta_j]=1\rangle$. Here $\alpha_1,\beta_1,\dots,\alpha_g,\beta_g$ are simple loops in $U_t$ which only intersect in $t_0$, and their homology classes in $H_1(C,\Z)$ form a symplectic basis.\par
If $f:\widetilde{C}\to C$ is a  ramified $H$-Galois cover with branch locus $t$, set $V=f^{-1}(U_t)$. Then $f|_V:V\to U_t$ is an unramified Galois covering. Then there is an epimorphism $\theta:\pi_1(U_t, t_0)\to H$. Conversely, such an epimorphism determines a ramified Galois covering of $C$ with branch locus $t$.  The order $m_i$ of $\theta(\gamma_i)$ is called the \emph{local monodromy datum} of the branch point $t_i$. Let $m=(m_1,\dots, m_r)$. The collection $(m,H,\theta)$ is called a \emph{datum}. The Riemann-Hurwitz formula implies that the genus $\widetilde{g}$ of the curve $\widetilde{C}$ is equal to
\begin{equation}
2\widetilde{g}-2=|H|(2g-2+\sum_{i=1}^r (1-\frac{1}{m_i}))
\end{equation}

We introduce the stack $R(H,g,r)$: The Objects of $R(H,g,r)$ are couples $((C, x_1,\dots, x_r),f:\widetilde{C}\to C)$ such that

\begin{enumerate}
\item $(C, x_1,\dots, x_r)$ is a smooth projective $r$-pointed curve of genus $g$.
\item $f:\widetilde{C}\to C$ is a finite cover, $H$ acts on  $\widetilde{C}$ and $f$ is $H$-invariant.
\item the restriction $f^{gen}:f^{-1}(C\setminus\{x_1,\dots, x_r\})\to C\setminus\{x_1,\dots, x_r\}$ is an \'etale $H$-torsor.
\end{enumerate}
Sometimes, for simplicity, we denote an element $((C, x_1,\dots, x_r),f:\widetilde{C}\to C)$ of $R(H,g,r)$ just by $(C,f)$. Note that $r=0$ is also possible which amounts to say that the covers $\widetilde{C}\to C$ are unramified. Moreover since our problem is insensitive to level structures, we may actually consider $R(H,g,r)$ as a coarse moduli space. As a result, we omit any assumptions on the automorphism group of the base curve $C$ whose non-triviality can be remedied either by considering the moduli stack or by imposing level structures.\par To each element $((C, x_1,\dots, x_r),f:\widetilde{C}\to C)$, we associate a Prym variety $P(\widetilde{C}/C)$ and obtain a Prym map 
$\mathscr{P}:R(H,g,r)\to A_{l,D}$, where $A_{l,D}$ denotes the moduli space of abelian varieties of dimension $l$ and polarization type $D$. Then we consider the following situation: Let $\widetilde{C}\to T$ be a family such that each $\widetilde{C}_t$ is an abelian $\widetilde{G}$-cover of $\P^1$. Let $H\subset\widetilde{G}$ be a normal subgroup. Then the quotients $\widetilde{C}_t\to C_t$ give rise to a subvariety $Z\subset R(H,g,r)$. On the other hand, the Prym variety $P(\widetilde{C}_t/C_t)$ can be defined as mentioned in the above and we are interested in the Zariski closure of the image $\mathscr{P}(Z)$ as a subvariety of $A_{l,D}$. Inspired by the results of Frediani \cite{F} about special subvarieties arising from families of abelian covers on the one hand and the results of Naranjo and Ortega \cite{no}, which assert that the Prym map $\mathscr{P}(\Z_2,g,r)$ is an embedding if $g>0$ and $r>6$, on the other hand, we prove that if the reduction of the Prym map $\mathscr{P}$ to characteristic $p$ is injective, then the assumption that the above subvariety is special is very restrictive, i.e., they can only be of the form $P(\widetilde{G})$ which is a special subvariety constructed using the group action of the family, see Section ~\ref{Main results} for more precise statements. This is in analogy with the similar problem for the Torelli locus which has been treated for example in \cite{FGP}, \cite{MZ} and \cite{M10}.

\section{the prym locus and abelian covers} \label{prym and abelian}
Let us denote the Jacobians of the curves $\widetilde{C}$ and $C$ respectively by $\widetilde{J}$ and $J$. Note that by definition, if $R$ is a Riemann surface,
\[J(R)=\jac(R)=H^0(R,\omega_{R})^*/H_1(R,\mathbb{Z}).\]
Since the finite group $H$ acts on $\widetilde{C}$ it also acts on the space of differential 1-forms $H^0(\widetilde{C},\omega_{\widetilde{C}})$ and $H_1(\widetilde{C},\mathbb{Z})$ and hence on the Jacobian $\widetilde{J}$. In particular, we denote by $\widetilde{J}^{H}$ the subgroup of fixed points of $\widetilde{J}$ under the action of $H$. The following theorem is proven in \cite{RR} (repectively, Theorem 2.5 and Proposition 3.1).

\begin{theorem} \label{Subtorus-jacobian}
\begin{enumerate}
\item $f^*J=(\widetilde{J}^H)^0$.\item The map $f$ induces an isogeny $J\times P(\widetilde{C}/C)\sim \widetilde{J}$
\end{enumerate}
\end{theorem}
We note that the isogeny mentioned in Theorem \ref{Subtorus-jacobian} is given by
\begin{gather}\label{isogeny equation}
\phi:J\times P(\widetilde{C}/C)\to \widetilde{J}\nonumber\\
\phi(c,\widetilde{c})=f^*c+\widetilde{c}
\end{gather}
For a Galois covering $f:\widetilde{C}\to C$ with $((C, x_1,\dots, x_r),f:\widetilde{C}\to C)\in R(H,g,r)$ and $\deg(f)=n$, one can compute the genus $\widetilde{g}:=g(\widetilde{C})$ by the Riemann-Hurwitz formula. Using the isogeny $f^*J\times P(\widetilde{C}/C)\sim \widetilde{J}$ we see that the dimension of the Prym variety $P(\widetilde{C}/C)=P(f)$ is equal to $l=\widetilde{g}-g$. The canonical principal polarization on $\widetilde{J}$ restricts to a polarization of type $D=(1,\dots,1,n,\dots,n)$ where $1$ occurs $g-1$ times and $n$ occurs $l-(g-1)$ times if $r=0$ and $1$ occurs $g$ times and $n$ occurs $l-g$ times otherwise.\par
Note that, it follows from Theorem \ref{Subtorus-jacobian} that if $C\cong\P^1$, then the Prym variety $P(\widetilde{C}/C)$ is isogeneous to the Jacobian $\widetilde{J}$. We will use this point in the sequel to deduce that some families are special. \par
Let $A_{l,D}$ denote the moduli space of complex abelian varieties of dimension $l$ and polarization type $D$. More precisely, $A_{l,D}=\H_l/\Gamma_D$ is the moduli space of polarized abelian varieties of type $D$ where \par  $\H_l:=\{M\in M_l(\C)\mid ^tM=M, \ima M\geq 0\}$ is the \emph{Siegel upper half space  of genus $l$} and

\[\Gamma_D=\{R\in\gl_{2l}(\Z)\mid R\begin{pmatrix}

0  &D\\

-D &0
\end{pmatrix}  {^tR}=\begin{pmatrix}

0  &D\\

-D &0
\end{pmatrix}\}\]
is an arithmetic subgroup. The above constructions behave well also in the families of curves and hence we obtain a morphism
\begin{equation}
\mathscr{P}=\mathscr{P}(H,g,r):R(H,g,r)\to A_{l,D}.
\end{equation}
We call the map $\mathscr{P}$ the \emph{Prym map of type $(H,g,r)$}. Our objective in this paper is to study the image of this map. The Prym map is even in the classical case known to be non-injective which implies that one needs to study other closely related aspects, namely the generic injectivity. \par
By the above mentioned $H$-action on $H^0(\widetilde{C},\omega_{\widetilde{C}})$ and
$H_1(\widetilde{C},\mathbb{Z})$,  we set:
\begin{align} \label{plus minus}
H^0(\widetilde{C},\omega_{\widetilde{C}})_+=H^0(\widetilde{C},\omega_{\widetilde{C}})^{H}(\cong H^0(C,\omega_{C})) \text{ and }\\
H^0(\widetilde{C},\omega_{\widetilde{C}})_-=H^0(\widetilde{C},\omega_{\widetilde{C}})/H^0(\widetilde{C},\omega_{\widetilde{C}})_+=
\bigoplus\limits_{\chi\in\irr(H)\setminus\{1\}}H^0(\widetilde{C},\omega_{\widetilde{C}})^{\chi}
\end{align}
Notice that $H^0(\widetilde{C},\omega_{\widetilde{C}})=H^0(\widetilde{C},\omega_{\widetilde{C}})_+\oplus H^0(\widetilde{C},\omega_{\widetilde{C}})_-$. \par The following lemma is then an immediate consequence of Theorem \ref{Subtorus-jacobian} above.
\begin{lemma} \label{prymvar}
Let $f:\widetilde{C}\to C$ be a Galois covering, then
\begin{equation} \label{Prymdef}
P(\widetilde{C}/C)={H^0(\widetilde{C},\omega_{\widetilde{C}})_-}^*/H_1(\widetilde{C},\mathbb{Z})_-
\end{equation}
\end{lemma}
\begin{definition}\label{Prym datum general}
A \emph{Prym datum} (of type $(H,g,r)$, compare \cite{CFGP}, Definition 3.1) is a triple $(\tilde{G}, \tilde{\theta}_s,H)$ where $\tilde{G}$ is a finite group, $\widetilde{\theta}_s:\Gamma_s\to\widetilde{G}$ is an epimorphism as above and $H$ is a normal subgroup of $\tilde{G}$, such that the quotient $f:\widetilde{C}\to C=\widetilde{C}/H$ is in $R(H,g,r)$.
\end{definition}
Let $\widetilde{G}$ be a finite group and let $\widetilde{C}\to \mathbb{P}^{1}$  be a $\widetilde{G}$-Galois covering of $\mathbb{P}^{1}$ with the Prym datum $(\tilde{G}, \tilde{\theta}_s,H)$. Set $V=H^0(\widetilde{C},\omega_{\widetilde{C}})$ and let $V=V_{+}\oplus V_{-}$ be the decomposition into $H$-invariant and $H$-anti-invariant parts as above. There is also the corresponding Hodge decomposition $H^1(\widetilde{C},\mathbb{C})_-=V_{-}\oplus \overline{V}_{-}$. Set $\Lambda=H_1(\widetilde{C},\mathbb{Z})_-$. The associated Prym variety is by definition the following abelian variety.
\begin{equation}\label{Prymdef1}
P(\widetilde{C}/C)=V^{*}_{-}/\Lambda,
\end{equation}
see \cite{BL} for more details.
\begin{remark} \label{smallestshimura}
There is a $\mathbb{Q}$-variation of Hodge structures over $T$ with fibers given by $H^1(\tilde{C_t},\mathbb{Q})_-$. We choose a Hodge-generic point $t_0\in T(\mathbb{C})$ and let $M\subset
\gl(H^1(\tilde{C_{t_0}},\mathbb{Q})_-)$ be the generic Mumford-Tate group of the family. Let $P_{f}$ be the natural Shimura variety associated to the reductive group $M$. So in general, this subvariety is different from the one with the same notation in \cite{MZ}, Remark 2.7. The special subvariety $P_{f}$ is the \emph{smallest} special subvariety that contains $Z$ and its dimension depends on the real adjoint group $M^{ad}_{\mathbb{R}}$. Indeed, if $M^{ad}_{\mathbb{R}}= Q_{1}\times...\times Q_{r}$ is the decomposition of $M^{ad}_{\mathbb{R}}$ to $\mathbb{R}$-simple groups then $\dim P_{f}= \sum \delta(Q_{i})$. If $Q_{i}(\mathbb{R})$ is not compact then $\delta(Q_{i})$ is the dimension of the corresponding symmetric space associated to the real group $Q_{i}$ which can be read from Table V in \cite{H}. If $Q_{i}(\mathbb{R})$ is compact (in this case $Q_{i}$ is called anisotropic) we set $\delta(Q_{i})=0$. We remark that for $Q=\PSU(p,q)$, $\delta(Q)=pq$ and for $Q=Psp_{2p}$, $\delta(Q)=\frac{p(p+1)}{2}$. Our computations below show that in fact such factors do occur in the decomposition of $M$, see Lemma~\ref{dimeigspace}. Note that $Z$ is a Shimura subvariety if and only if $\dim Z=\dim P_{f}$, i.e., if and only if $Z=P_{f}$.
\end{remark}
These observations lead to the following key lemma.
\begin{lemma} \label{dimP_f}
If $\dim P_f>s-3$, then the family does not give rise to a special subvariety of the Prym locus.
\end{lemma}
\begin{proof}
By the constructions and explanations in previous paragraphs, we have a map $\mathscr{P}:R_g\to A_{g-1}$ (resp. $R_{g,2}\to A_g$) and $Z=\overline{\mathscr{P}(T)}\subset A_{g-1}$ (resp. $A_{g}$). Now unlike the Torelli map, the Prym map is not injective, however, it holds that $\dim Z\leq s-3$. Hence if $\dim P_f>s-3$, one concludes that $Z\neq P_f$ and therefore $Z$ is not a special subvariety by the above.
\end{proof}
In the light of the above lemma, our strategy is to show that for families with large $s$, the subvariety $P_f$ constructed above is of dimension strictly greater than $s-3$, hence the subvariety $Z$ is not special by the above lemma.
\subsection{Abelian covers of $\mathbb{P}^{1}$ their invariants}
An abelian Galois cover of $\mathbb{P}^{1}$ is determined by a
collection of equations in the following way: Consider an $m\times
s$ matrix $A=(r_{ij})$ whose entries $r_{ij}$ are in
$\mathbb{Z}/N\mathbb{Z}$ for some $N\geq 2$. Let
$\overline{\mathbb{C}(z)}$ be the algebraic closure of
$\mathbb{C}(z)$. For each $i=1,...,m,$ select a function $w_{i}\in
\overline{\mathbb{C}(z)}$ with
\begin{equation} \label{equation abelian}
w_{i}^{N}=\prod_{j=1}^{s}(z-z_{j})^{\widetilde{r}_{ij}} \text{ for  } i=
1,\dots, m, 
\end{equation}
in $\mathbb{C}(z)[w_{1},...,w_{m}]$. Note that $z_j\in
\mathbb{C}$. In other words, there exists a projective non-singular curve $Y$
birational to the affine curve defined by the above equations together with a 
covering map $\pi : Y\to \mathbb{P}^{1}$ with abelian Galois group. 

Here $\widetilde{r}_{ij}$ is the lift of $r_{ij}$ to $\mathbb{Z}
\cap [0,N)$. We impose the condition that the sum of the columns
of $A$ are zero. This implies that the cover is \emph{not}
ramified over infinity. The matrix $A$ will be called the matrix
of the covering. We also consider the row and column spans of $A$ as modules
over the ring $\mathbb{Z}/N\mathbb{Z}$ and so all of the
operations with rows and columns will be carried out in the ring
$\mathbb{Z}/N\mathbb{Z}$, i.e., it will be considered modulo $N$.
The abelian Galois group $G$ of the covering is isomorphic to the
column span of the matrix $A$ and hence can be considered as a
subgroup of $(\mathbb{Z}/N\mathbb{Z})^{m}$ (denoted also by
$\mathbb{Z}_{N}^{m}$).
\begin{remark} \label{abeliangroupcharacter}
Let $\widetilde{G}$ be a finite abelian group, then the character group of $\widetilde{G}$, $\mu_{\widetilde{G}}=\Hom(\widetilde{G},\mathbb{C}^{*})$ is isomorphic to $\widetilde{G}$. To see this, first assume that $\widetilde{G}=\mathbb{Z}/N$ is a cyclic group. Fix an isomorphism between $\mathbb{Z}/N$ and the group of $N$-th roots of unity in $\mathbb{C}^{*}$ via $1\mapsto\exp(2\pi i/N)$. Now the group $\mu_{\widetilde{G}}$ is isomorphic to this latter group via $\chi\mapsto \chi(1)$. In the general case, $\widetilde{G}$ is a product of finite cyclic groups, so this isomorphism extends to an isomorphism $\varphi_{\widetilde{G}}: \widetilde{G} \xrightarrow{\sim} \mu_{\widetilde{G}}$. In the sequel, we use this isomorphism frequently to identify elements of $\widetilde{G}$ with its characters without referring to the isomorphism $\varphi_{\widetilde{G}}$.
\end{remark}

For our applications, with notation as in the previous pages, we fix an isomorphism of $\widetilde{G}$ with a product of $\mathbb{Z}/n\Z$'s and an embedding of $\widetilde{G}$ into $(\mathbb{Z}/N\Z)^{m}$.\par Furthermore we set $\tilde\alpha_j = \sum_{i=1}^m n_i \tilde{r}_{ij} \in \Z$ (note that $\tilde\alpha_j$ is not necessarily in $\Z\cap [0,N)$).

Let us denote by $\omega_X$ the canonical sheaf of $X$.

Similar to the case of $\pi_{*}(\mathcal{O}_X)$, the sheaf $\pi_{*}(\omega_X)_{\chi}$ decomposes according to the action of $\widetilde{G}$.

For the line bundles $L_{\chi}$ corresponding to the character $\chi$ associated to the element $a\in \widetilde{G}$ and $\pi_{*}(\omega_X)_{\chi}$ we have the following result proven in \cite{MZ}.

\begin{lemma} \label{eigenbundleformula}
Notation being as above,  $L_{\chi}=\mathcal{O}_{\mathbb{P}^{1}}(\displaystyle\sum_{1}^{s}\langle\frac{\tilde\alpha_j}{N}\rangle)$, where $\langle x\rangle$ denotes the fractional part of the real number $x$ and
\[\pi_{*}(\omega_X)_{\chi}= \omega_{\mathbb{P}^{1}} \otimes L_{\chi^{-1}}=\mathcal{O}_{\mathbb{P}^{1}}(-2+\sum_{1}^{s}\langle -\frac{\tilde\alpha_j}{N}\rangle).\]
\end{lemma}
Let $n\in\widetilde{G}$ be the element $(n_1,\dots, n_m)\in\widetilde{G}\subset (\mathbb{Z}/N\mathbb{Z})^{m}$. By Lemma \ref{eigenbundleformula}, $\dim H^0(\widetilde{C},\omega_{\widetilde{C}})_{n}=-1+\displaystyle \sum_{j=1}^{s}\langle-\frac{\tilde\alpha_j}{N}\rangle$. A basis for the $\mathbb{C}$-vector space $H^0(\widetilde{C},\omega_{\widetilde{C}})$ is given by the forms
\begin{equation} \label{basis}
\omega_{n,\nu}=z^{\nu} w_{1}^{n_1}\cdots w_{m}^{n_m}\displaystyle \prod_{j=1}^{s} (z-z_j)^{\lfloor -\frac{\tilde\alpha_j}{N}\rfloor}dz.
\end{equation}

Here $0\leq\nu\leq -1+\displaystyle \sum_{j=1}^{s}\langle-\frac{\tilde\alpha_j}{N}\rangle$. The fact that the above elements constitute a basis can be seen in \cite{MZ}, proof of Lemma 5.1, where the dual version for $H^1(C,\mathcal{O}_C)$ is proved. \par

The general method of our later computations in Section ~\ref{Main results} is as follows: We remark that if $n=(n_1,\dots,n_m)\in\widetilde{G}=\mathbb{Z}_{d_1}\times\cdots\times\mathbb{Z}_{d_m}\subset(\mathbb{Z}/N\mathbb{Z})^{m}$, we consider the $n_i\in [0,N)$ and their sum as integers. \par The action of the abelian subgroup $H$ is naturally inherited from that of $\tilde G$ and the latter is described as follows: Let $g=(g_1,\dots, g_m)\in\tilde G$ and write $\ord g_i=v_i$. Then the action of $g$ on each $w_i$ is given by $g\cdotp w_i=\xi_{v_i}w_i$, where $\xi_{v_i}$ denotes a $v_i$-th primitive root of unity. 

With this notation, $H^0(\widetilde{C},\omega_{\widetilde{C}})_+$, i.e., the group of $H$-invariant differential forms is the set of all $\omega_{n,\nu}$ with $\sum n_i/a_i\in\mathbb{Z}$ for all $h=(h_1,\dots,h_m)\in H$ (with $a_i=\ord h_i$).

The eigenspace $H^0(\widetilde{C},\omega_{\widetilde{C}})_-$ is then given by the complement, i.e., the set of all $\omega_{n,\nu}$ for whom there exists $h=(h_1,\dots,h_m)\in H$ such that $\sum n_i/a_i\notin\mathbb{Z}$. \par The following lemma generalizes \cite{M21}, Lemma 2.5.

\begin{lemma} \label{dimeigspace}
Let the notation be as above. The group $\widetilde{G}$ acts on the space $H^0(\widetilde{C},K_{\widetilde{C}})_{-}$ and for $g\in\widetilde{G}$, it holds that
\begin{align}
H^0(\widetilde{C},K_{\widetilde{C}})_{-,g}=\begin{cases}
H^0(\widetilde{C},K_{\widetilde{C}})_{g}, & \text{if } \exists h=(h_1,\dots,h_m)\in H \text{ such that }\sum g_i/a_i\notin\mathbb{Z}, \text{ with }a_i=\ord h_i,\\
0 & \text{otherwise.}
\end{cases}
\end{align}
Similar statements hold for $H^1(\widetilde{C},\C)_{-,g}$.
\end{lemma}
\begin{proof}
All of the claims follow directly using the basis for $H^0(\widetilde{C},K_{\widetilde{C}})$ and the action of $\widetilde{G}$ on $H^0(\widetilde{C},K_{\widetilde{C}})_-$  decribed above.
\end{proof}

Families of abelian $\widetilde{G}$-covers of $\P^1$ can be constructed as follows:  Let $T_{s}\subset (\mathbb{A}_{\C}^{1})^{s}$ be the complement of the big diagonals, i.e., $T_{s}= \{(z_{1},\dots,z_{s})\in(\mathbb{A}_{\C}^{1})^{s}\mid z_{i}\neq z_{j} \forall i\neq j \}$. Over this affine open set we define a family of abelian covers of $\mathbb{P}^{1}$ by the equation \ref{equation abelian} with branch points $(z_{1},\dots,z_{s})\in T_{s}$ and $\widetilde{r}_{ij}$ the lift of $r_{ij}$ to $\mathbb{Z}\cap[0,N)$ as before. Varying the branch points we get a family $f:\tilde{\sC}\to T_s$ of smooth projective curves over $T_s$ (viewed as a complex manifold of dimension $s-3$) whose fibers $\tilde{C}_t$ are abelian covers of $\mathbb{P}^{1}$ introduced above.

\begin{remark} \label{eigenspacetype}
Let $f:\tilde{\sC}\rightarrow T$ be a family of abelian Galois covers of $\mathbb{P}^{1}$ as constructed in section $1$. Then the local system $\mathcal{L}=R^{1}f_{*}\mathbb{C}_-$ gives rise to a polarized variation of Hodge structures (PVHS) of weight $1$ whose fibers are the HS discussed above. Consider the associated monodromy representation $\pi_{1}(T,x)\to \gl(V)$, where $V$ is the fiber of $\mathcal{L}$ at $x$. The Zariski closure of the image of this morphism is called the \emph{monodromy group} of $\mathcal{L}$. We denote the identity component of this group by $\mon^{0}(\mathcal{L})$. The PVHS decomposes according to the action of the abelian Galois group $G$ and the eigenspaces $\mathcal{L}_{i}$ (or $\mathcal{L}_{\chi}$ where $i\in G$ corresponds to character $\chi\in \mu_{G}$ by Remark~\ref{abeliangroupcharacter}) are again variations of Hodge structures and we are mainly interested in these. Take a $t\in T$ and assume that $h^{1,0}((\mathcal{L}_{i})_{t})=a$ and $h^{0,1}((\mathcal{L}_{i})_{t})=b$.  The above computations show how to calculate $h^{1,0}((\mathcal{L}_{i})_{t})$ (resp. $h^{0,1}((\mathcal{L}_{i})_{t})$). Since monodromy group respects the polarization of the Hodge structures (\cite{R}, 3.2.6), $(\mathcal{L}_{i})_{t}$ is equipped with a Hermitian form of signature $(a,b)$ (see \cite{DM}, 2.21 and 2.23). This implies that $\mon^{0}(\mathcal{L}_{i}) \subseteq U(a,b)$. In this case, we say that $\mathcal{L}_{i}$ is \emph{of type} $(a,b)$. Lemma~\ref{dimeigspace} (together with \cite{MZ}, Proposition 2.8) computes the type of any eigenspace. Two eigenspaces $\mathcal{L}_{i}$ and $\mathcal{L}_{j}$ of types $(a,b)$ and $(a^{\prime},b^{\prime})$ are said to be \emph{of distinct types} if $\{a,b\}\neq \{a^{\prime},b^{\prime}\}$.  We call an eigenspace $\mathcal{L}_{i}$ \emph{trivial} if it is of type $(a,0)$ or $(0,b)$.
\end{remark}
\begin{remark} \label{monodromydecomp}
Let $\mathcal{V}$ be a variation of Hodge structures over a non-singular connected complex algebraic variety. If there is a point $s$, such that the Mumford-Tate group $\mt_s$ is abelian, then the connected monodromy group is a normal subgroup of the generic Mumford-Tate group $M$. In fact in this case $\mon^0=M^{der}$, see \cite{Andr}. In particular, if $Z\subset A_g$ is special, then $\mon^0=M^{der}$. Consequently, if the family $f\colon Y\to T$ gives rise to a Shimura subvariety and $M^{ad}_{\mathbb{R}}=\prod_{1}^{l} Q_{i}$ as a product of simple Lie groups then $\mon^{0,ad}_{\mathbb{R}}=\prod_{i\in K} Q_{i}$ for some $K\subset \{1,\dots, l\}$.
\end{remark}
\section{Shimura families in the Prym locus} \label{shimura families}
Recall from Section \ref{prym and abelian} that $A_{l,D}=\H_l/\Gamma_D$ is the moduli space of polarized abelian varieties of type $D$, where $\H_l:=\{M\in M_l(\C)\mid ^tM=M, \im M\geq 0\}$, is the Siegel upper half space  of genus $l$ and $\Gamma_D=\{R\in\gl_{2l}(\Z)\mid R\begin{pmatrix}
0  &D\\

-D &0
\end{pmatrix}  {^tR}=\begin{pmatrix}
0  &D\\

-D &0
\end{pmatrix}\}$ is an arithmetic subgroup. Note that $\H_l=\gsp_{2l}(\R)/K$, where $\gsp_{2l}$ is the standard $\Q$-group of symplectic similitudes on the standard symplectic $\Q$-space $\Q^{2l}$ and $K$ is a maximal compact subgroup. So $A_{l,D}$ can be written as a double quotient $\Gamma_D\backslash\gsp_{2l}(\R)/K$. Such double quotients are called Shimura variety and their structure has beeen studied extensively. A special (or Shimura) subvariety of $A_{l,D}$ is then an algebraic subvariety of the form $Y=\Gamma\backslash\D\hookrightarrow A_{l,D}$ induced by an injective homomorphism $\L\hookrightarrow \Sp_{2l}(\R)$ of algebraic groups. In particular, it is a totally geodesic subvariety.\par Let $\widetilde{G}$ be a finite group and consider a family $\widetilde{\sC}\to T_s$ whose fibers $\widetilde{C}_t$ are $\widetilde{G}$-Galois coverings of $\mathbb{P}^{1}$ with a fixed Prym datum $\Sigma\coloneqq(\tilde{G}, \tilde{\theta}_s,H)$. Associating to $t\in T_s$ the class of the pair $((C_t,x_1,\dots,x_r),\pi_t:\widetilde{C}_t\to C_t)$ gives a map $T_s\to R(H,g,r)$ with discrete fibers. We denote the image of this map by $R(\Sigma)$. It follows that $R(\Sigma)$ is a subvariety of dimension equal to $s-3$, see also \cite{CFGP}, p. 6. As in the last section, set $V_t=H^0(\widetilde{C}_t,\omega_{\widetilde{C}_t})$ and let $V_t=V_{+,t}\oplus V_{-,t}$ be the decomposition under the action of $H$. There is also the corresponding Hodge decomposition $H^1(\widetilde{C}_t,\mathbb{C})_-=V_{-,t}\oplus \overline{V}_{-,t}$. Set $\Lambda_t=H_1(\widetilde{C}_t,\mathbb{Z})_-$. The associated Prym variety is by \ref{Prymdef1}, $P(\widetilde{C}_t/C_t)=V^{*}_{-,t}/\Lambda_t$, an abelian variety of dimension $l=\widetilde{g}-g$. So we have a map $R(\Sigma)\xrightarrow{\mathscr{P}} A_{l,D}$.\par In this paper, we are interested in determining whether the subvariety $Z=\overline{\mathscr{P}(R(\Sigma))}\subset A_{l,D}$ is a special or Shimura subvariety.\par We remark that one can, possibly after replacing $T_s$ with a suitable finite cover, endow the abelian scheme with a level structure and hence consider the resulting map $T_s\to A_{l,D,n}$ to the fine moduli space. Note that the answer to the above question is independent of the level structure or the polarization. Alternatively, we can consider $A_{l,D}$ as a coarse moduli space. The Prym varieties of the fibers of the family $\widetilde{C}_t\to T_s$ fit into a family $P\rightarrow T_s$ which is an abelian scheme over $T_s$ that admits naturally an action of the group ring $\mathbb{Z}[\widetilde{G}]$. This action defines a Shimura subvariety of PEL type $P(\widetilde G)$ in $A_{l,D}$ (or in the stack $\mathcal{A}_{l,D}$) that contains $Z$. The following constrcution of the subvariety $P(\widetilde G)$ is adapted for the case of  Prym varieties from \cite{M10}. See also \cite{FGP} for a different approach. Fix a base point $t \in T_s$ and let $(P_t,\lambda)$ be the corresponding Prym variety with $\lambda$ as its polarization of type $D$. Let $(V_{\mathbb{Z}},\psi)$ be as above. We fix a symplectic similitude $\sigma :H^{1}(P_t, \mathbb{Z})\to V_{\mathbb{Z}}$. Let $F=\mathbb{Q}[\widetilde{G}]$. The group $\widetilde{G}$ acts on $H^0(\widetilde{C}_t,\omega_{\widetilde{C}})_-$ and thereby on the Prym variety $P(\widetilde{C}_t/C_t)$. We therefore view $H^0(\widetilde{C}_t,\omega_{\widetilde{C}})_-$ as an $F$-module. Via $\sigma$, the Hodge structure on $H^1(P_t,\mathbb{Q})=H^1(\widetilde{C}_t,\mathbb{Q})_-$ corresponds to a point $y\in\H_l$ and one obtains the structure of an $F$-module on $V_{\mathbb{Q}}$. $F$ is isomorphic to a product of cyclotomic fields and is equipped with a natural involution $*$ which is complex conjugation on each factor. The polarization $\psi$ on $V_{\mathbb{Q}}$ satisfies
\[\psi(bu,v)=\psi(u,b^{*}v) \text{ for all } b\in F \text{ and } u,v \in V.\]
Let us define the subgroup:
\begin{equation}\label{subgroup}
N= \gsp(V_{\mathbb{Q}}, \psi)\cap \gl_{F}(V_{\mathbb{Q}}).
\end{equation}
If $h_{0}: \mathbb{S}\rightarrow \gsp_{2l, \mathbb{R}}$ is the Hodge structure on $V_{\mathbb{Z}}=H^{1}(P_{t},\mathbb{Z})$ corresponding to the point $y \in \H_l$, then by the above $F$-action this homomorphism factors through the subvariety $N_{\mathbb{R}}$. 
As $A_{l,D}$ has the structure of a Shimura variety, one can talk about its Shimura (or special) subvarieties. Let $L=\gsp_{2l}$. Define the subset $Y_N\subseteq\H_l$ as follows.
\[Y_N=\{h: \mathbb{S}\rightarrow L_{2l, \mathbb{R}}| h \text{ factors through } N_{\mathbb{R}}\}.\]
The point $y$ lies in $Y_{N}$ and there is a connected component $Y^{+}\subseteq Y_{N}$ which contains $y$. We define $P(\widetilde{G})$ to be the image of $Y^{+}$ under the map
\[\H_l\to L(\mathbb{Z})\setminus\H_l\cong L(\mathbb{Q})\setminus\H_l\times
L(\mathbb{A}_f)/L(\widehat{\mathbb{Z}})\cong A_{l,D}(\mathbb{C}).\]
If $Y^{+}$ is a connected component of $Y_N$ and $\gamma K_n\in L(\mathbb{A}_f)/K_n$, the image of $Y^{+}\times \{\gamma K_n\}$ in $A_{l,D}$ is an algebraic subvariety. We define a \emph{Shimura subvariety} as an algebraic subvariety $S$ of $A_{l,D}$ which arises in this way, i.e., there exists a connected component $Y^+\subset Y_N$ and an element $\gamma K_n\in L(\mathbb{A}_f)/K_n$ such that $S$ is the image of $Y^{+}\times\{\gamma K_n\}$ in $A_{l,D}$.\par For $t=(z_{1},\dots,z_{s})\in T_s$, let $((C_t,x_1,\dots,x_r),\pi_t:\widetilde{C}_t\to C_t)\in R(H,g,r)$ be the covering corresponding to
$t$. For this $t$, consider the Hodge decomposition $H_1(\widetilde{C}_t,\mathbb{C})_-=V_{-,t}\oplus\overline{V}_{-,t}$ which corresponds to a complex structure on
$H_1(\widetilde{C}_t,\mathbb{R})_-$. We therefore get a point $f(t)\in\H_{l}$. Indeed we obtain a morphism $f:T_s\to\H_{l}$ and the following commutative diagram.
\begin{equation} \label{normal diag}
\begin{tikzcd} 
T_s \arrow{r}{f} \arrow[swap]{d}{\iota^0} & \H_{l} \arrow{d}{\iota} \\
R(\Sigma) \arrow{r}{\mathscr{P}} & A_{l,D}
\end{tikzcd}
\end{equation}
It follows by construction of $P(\widetilde{G})$ that $Z\subseteq P(\widetilde{G})$. As we remarked earlier, the Prym map is not in general injective. In order to conclude the equality $Z= P(\widetilde{G})$ and hence the speciality of $Z$, we still need to assure that the differential of the Prym map on $R(\Sigma)$ is injective, whence $\dim R(\Sigma)=\dim \mathscr{P}(R(\Sigma))$. 

The following lemma computes $\dim P(\widetilde{G})$ and so is very useful in te sequel. 
\begin{lemma} \label{dim P(G)}
Let $d_n=H^{1,0}( P(\widetilde{G}))_n=H^0(\widetilde{C},\omega_{\widetilde{C}})_{-,n}$, then 
\[\dim P(\widetilde{G})= \sum_{2n\neq 0} d_{n}d_{-n}+\frac{1}{2}\sum_{2n=0}d_{n}(d_{n}+1).\] 
\end{lemma}
Note that $2.0=0$ in $\widetilde G$, so in fact the second sum in the right hand side of the above equality is always meaningful and if $|\widetilde{G}|$ is an odd number it will be zero.

\section{ordinary Prym varieties} \label{ordinary prym}
An abelian variety of dimension $g$ over a field $k$ of characteristic $p$ is called ordinary if $A[p](k)\simeq(\Z/p\Z)^g$. The Hodge filtration $\Fil:=F^1_{Hodge}\subset H^1_{\dR}(A/k)$ is well-known to be the kernel of the action of the Frobenius. Let $W$ denote the ring of Witt vectors and $\mathbf{A}$ a lift of $A$ to $W$. Then the crystalline cohomology $H^1_{\cris}(\mathbf{A}/W)$ is a lift of $H^1_{dR}(A/k)$ to characteristic zero, i.e. $H^1_{\cris}(\mathbf{A}/W)\otimes k\simeq H^1_{dR}(A/k)$. Let $W_n=W/p^nW$ be the ring of truncated Witt vectors. Then there is a bijective correspondence between the formal liftings of $A/k$ to $W_n$ and liftings of $\Fil_A$ to a direct summand of $H^1_{\cris}(\mathbf{A}/W_n)$, see \cite{DO}, Theorem 1.3. If $A$ is an ordinary abelian variety, there is a distinguished lifting $\Fil_{\can}$ of $\Fil$, which by the above explanation, defines a lifting $\mathbf{A}_{\can}$ of $A$ to $W_n$ called the \emph{canonical lifting}.

 A smooth algebraic curve $C$ over a field $k$ of characteristic $p$ is called ordinary if the Jacobian $\jac(C)$ is an ordinary abelian variety. A cover $\widetilde{C}\to C$ in $R(H,g,r)$ is called Prym-ordinary, if the Prym variety $P(\widetilde{C}/C)$ is an ordinary abelian variety. In analogy with \cite{DO}, we make the following definition.
\begin{definition} 
A Prym-ordinary cover $\widetilde{C}\to C$ is called Prym-pre-$W_2$-canonical if there is a smooth curve $\mathbf{Y}/W_n$ and a cover $\widetilde{\mathbf{Y}}\to\mathbf{Y}$ whose Prym variety $P(\widetilde{\mathbf{Y}}/\mathbf{Y})$ is isomorphic as an abelian variety to $P(\widetilde{C}/C)_{\can}$.
\end{definition}
In \cite{DO}, Dwork and Ogus introduced an invariant $\beta$ with the following key property: if an ordinary curve $C$ is pre-$W_2$-canonical (meaning that the canonical lifting of its Jacobian is again a Jacobian), then there exists a lifting $\mathbf{C}$ such that $\Fil_{\mathbf{C}}=\Fil_{\can}$ and this implies that $\beta=0$. The invariant $\beta$ can also be defined in families and in this case in denoted by $\widetilde{\beta}$. We observe that if the Prym map is injective, then this invariant vanishes also for families of Prym varieties.
\begin{lemma}\label{Prympre}
Let the Prym map $\mathscr{P}(H,g,r)$ be injective. Let $(\widetilde{C}\to C)$ be Prym-ordinary, then it is Prym-pre-$W_2$-canonical if and only if there is a lifting $\widetilde{\mathbf{Y}}\to\mathbf{Y}$ (of $\widetilde{C}\to C$) such that $\Fil_{\widetilde{Y},-}=\Fil_{\can,-}$, in particular $\beta_{\widetilde{C}/C}=0$.
\end{lemma}
\begin{proof}
If the cover is Prym-pre-$W_2$-canonical, then there is a cover $\widetilde{\mathbf{Y}}\to\mathbf{Y}$ such that $P(\widetilde{\mathbf{Y}}/\mathbf{Y})=P(\widetilde{C}/C)_{\can}$. In particular, the reduction $P(\widetilde{Y}/Y)$ of $P(\widetilde{\mathbf{Y}}/\mathbf{Y})$ to $k$ is isomorphic to $P(\widetilde{C}/C)$. By the assumption on the Prym map, this implies that $[\widetilde{Y}\to Y]=[\widetilde{C}\to C]$, i.e., $\widetilde{\mathbf{Y}}\to\mathbf{Y}$ really is a lift of $\widetilde{C}\to C$. 
\end{proof}
Now let $\widetilde{f}:\widetilde{\sC}\rightarrow T$ be a family of abelian Galois covers of $\P^1_{\C}$ as constructed in section $1$. Let $R=\Z[1/N,u]/\Phi_N$ be the $N$th cyclotomic polynomial.  Let $f:\sC\to T$ be the quotient family. Note that $R$ can be embedded into $\C$ by sending the image of $u$ to $\exp(\frac{2\pi i}{N})$. We consider $T\subset(\A^1_R)^ s$ as the complement of the big diagonals, i.e., as the $R$-scheme of ordered $s$-tuples of distinct points in $\A^1_R$. For a prime number $p$,  we denote by $\sP$ a prime of $R$ lying above $p$. One can choose a prime number $p\equiv 1$ (mod $N$) and an open subset $U^{\prime}\subset T\otimes\F_p\cong T\otimes_R R/\sP$ such that for all $t\in U^{\prime}$, the fibers $C_t$ of $f$ are ordinary curves in characteristic $p$. This is possible for example by \cite{Bouw}, Theorem on page 2.  Since $\widetilde{C}_t\to C_t$ is a finite abelian cover, there exists an open subset $U\subset U^{\prime}$ such that both $\widetilde{C}_t$ and $C_t$ are ordinary for every $t\in U$. Because of the isogeny $\jac(\widetilde{C}_t)\sim\jac(C_t)\oplus P_t$, this implies that $P_t$ is also an ordinary abelian variety for $t\in U^{\prime}$. By Lemma~\ref{Prympre}, if $f:\widetilde{\sC}\to T$ is a family of abelian covers with a fixed Prym datum $\Sigma=(\tilde{G}, \tilde{\theta}_s,H)$ for which the Prym map is an immersion or generically injective, then there exist an open subset over which $\widetilde{\beta}_U=0$. In \cite{F}, under  the assumption that the differential of the restriction of the Prym map to the subvariety $Z$ is injective (equivalent to the surjectivity of the multiplication map), which implies that there is an immersion from $Z$ to $A_{l,D}$, it is proved that the image of $Z$ is not totally geodesic by applying the second fundamental form. On the other hand, the results of Naranjo and Ortega in \cite{no}, shows that the Prym map $\mathscr{P}(\Z_2, g, r)$ is an embedding if $g>0$ and $r\geq 6$. Here we make a similar assumption in characteristic $p$. However, we only need to assume the following.\\

(*) The reduction of the Prym map $\mathscr{P}=\mathscr{P}(H,g,r):R(H,g,r)\to A_{l,D}$ to characteristic $p$ is injective. \\

For $p$ and $U$ as above, consider the restricted family $C_U\to U$. The abelian group $\widetilde{G}$ also acts on the sheaves $\sL(C_U/ U)$ and gives the eigensheaf decomposition $\sL(C_U/ U)=\oplus_{n\in\widetilde{G}}\sL_{(n)}$. The same is true for $\E_U=\E(C_U/ U)$ and $\sK_U=\sK(C_U/ U)$. This in turn gives us the decomposition $\widetilde{\beta}_{C_U/ U}=\sum_n \widetilde{\beta}_n$. Here $\widetilde{\beta}_n$ is considered as a section of $F^*_U\sL_{(n)}$.

In particular, the subgroup $H\subset\widetilde{G}$ induces the decomposition $\widetilde{\beta}_{C_U/ U}=\widetilde{\beta}_{C_U/ U,+}\oplus\widetilde{\beta}_{C_U/ U,-}$, where $\widetilde{\beta}_{C_U/ U,+}$ (resp. $\widetilde{\beta}_{C_U/ U,-}$) denotes that $H$-invariant (resp. $H$-anti-invariant) part under the action of $H$. In the same way, we have eigenspaces $\sK_{U,-}=\sK(C_U/ U)_-$ and $\sL(C_U/ U)_-$. In particular, we have $-\nabla\widetilde{\beta}_{C/T,-}: F^{*}_{T}\mathcal{K}_- \rightarrow\Omega^{1}_{T/k}$ which is equal to the composition
 \[F^{*}_{T}\mathcal{K}_-\hookrightarrow F^{*}_{T}\Sym^{2}(\mathbb{E})_-\xrightarrow{\Sym^{2}(\gamma)}\Sym^{2}(\mathbb{E})_-\xrightarrow{\kappa} \Omega^{1}_{T/k}. (*)\]
The map $\kappa: \Sym^{2}(\mathbb{E})\rightarrow\Omega^{1}_{T/k}$ is the Kodaira-Spencer map associated to the family $\widetilde{f}:\widetilde{\sC}\rightarrow T$. Note that the group $\widetilde{G}$ also acts on the sheaves $\mathcal{K}_-, \Sym^{2}(\mathbb{E})_-$ and so on and we will denote the invariant susheaves by $\mathcal{K}^{\widetilde{G}}_-, \Sym^{2}(\mathbb{E}_{U})^{\widetilde{G}}_-$.

We use this to show that there are no more special subvarieties of the Prym locus obtained from families of Galois covers.  

\begin{lemma} \label{DWvanish}
For prime number $p$ and open subset $U$ as above, if the family satisfies condition (*) and gives rise to a Shimura subvariety $Z\subseteq A_{p, D}$ then for any $t\in U$ we have that the Prym variety $P_{t}=P(\widetilde{C}_t/C_t)$ is Prym-pre-$W_{2}$-canonical and in particular $\widetilde{\beta}_{C_{U}/U,-}=0$.
\end{lemma}
\begin{proof}
 This is an analogue of the results of \cite{Noo} for the Prym varieties. According to this result, a Shimura variety gives rise to a translation of a formal subtorus of local moduli. The group action then forces this to be a torus. In particular, if the moduli variety $Z$ is a Shimura subvariety and $t\in T$ is an ordinary point (i.e., $P(\widetilde{C}_t/C_t)$ is an ordinary abelian variety) then the canonical lifting $P^{can}_{t}$ of $P_{ t}$ is a $W(k)$-valued point of $Z$. This means in particular that it is a Jacobian and hence $P_{t}$ is pre-$W_{2}$- canonical. By Dwork-Ogus theory this forces $\widetilde{\beta}_{C_{U}/U,-}$ to be zero.
 \end{proof}
From now on we just work with the restricted family $C_{U}/U$ whose fibers are all ordinary instead of $C/T$ and denote it simply as $C/U$. 
\begin{proposition} \label{exactsequencevanish}
If the family of abelian covers gives rise to a Shimura subvariety in $A_{l,D}$, then the map
\[F^{*}_{U}\mathcal{K}^{\widetilde{G}}_-\hookrightarrow F^{*}_{U}\Sym^{2}(\mathbb{E}_{U})^{\widetilde{G}}_-\xrightarrow{\Sym^{2}(\gamma)} \Sym^{2}(\mathbb{E}_{U})^{\widetilde{G}}_-\xrightarrow{\mult^{\widetilde{G}}}f_{*}(\omega_{C/U}^{\otimes2})^{\widetilde{G}}_-\]
vanishes identically.
\end{proposition}
\begin{proof}
Since the fibers are Prym-ordinary curves over $U$, assuming that the family gives rise to a Shimura subvariety in $A_{g}$, it follows from the Lemma~\ref{DWvanish} that $\widetilde{\beta}_{C_{U}/U,-}=0$ and hence $\nabla \widetilde{\beta}_{C_{U}/U,-}=0$. On the other hand, by an argument similar to \cite{MZ}, Lemma 4.1, the above composition map is just $(\nabla \widetilde{\beta}_{C_{U}/U,-})^{\widetilde{G}}$. 
 \end{proof}
Next, we mention a lemma which allows us to compute explicitly the Hasse-Witt matrix of an abelian covering and whose proof can be found in \cite{MZ}, Lemma 5.1. 
By Proposition \ref{exactsequencevanish}, this lemma will be needed to
compute the obstruction $\widetilde{\beta}_{C/U}$. Consider a non-singular projective abelian cover  $\pi:\widetilde{C}\to \mathbb{P}^{1}$ with Galois group $\widetilde{G}$ and matrix $A$. Let $a=(a_{1},...,a_{m}) \in\widetilde{G} \subseteq \mathbb{Z}_{N}^{m}$ be an element in the Galois group of the abelian covering (or the
corresponding character $\chi$. See Remark~\ref{abeliangroupcharacter}). Consider
the tuple $(\sum_{1}^{m}a_{i}\widetilde{r}_{i1},...,\sum_{1}^{m}a_{i}\widetilde{r}_{is})=(\widetilde{\alpha}_{1},...,\widetilde{\alpha}_{s})$ as in Lemma~\ref{eigenbundleformula}. Take a prime number $p$ such that $p\equiv 1$ (mod $N$) and let $q= \frac{p-1}{N}$. With these notations, we have

\begin{lemma} \label{HasseWitt}
With notation as above, the eigenspaces $H^{1}(\widetilde{C},\mathcal{O}_{\widetilde{C}})_{\chi}$ are stable under the Hasse-Witt map and there is a basis
$(\xi_{a,i})_{i}$ on each eigenspace in which the $(i,j)$ entry of matrix is given by the formula:
\[\sum_{\sum l_{i}=\Upsilon}\binom{q.[-\alpha_{1}]_{N}}{l_{1}}...\binom{q.[-\alpha_{s}]_{N}}{l_{s}}z_{1}^{l_{1}}...z_{s}^{l_{s}},\]
where $\Upsilon= (d_{n}-i+1)(p-1)+(i-j)$ and
$\binom{a}{b}=\frac{a!}{b!(a-b)!}$.

\

\end{lemma}
\section{Main results} \label{Main results}
In \cite{F}, Frediani has shown that if the differential of the Prym map is a surjection and there exists an $n\in\widetilde{C}$ with $d_n\geq 2$ and $d_{-n}\geq 2$, then the family is not a totally geodesic family and hence is not Shimura. In this section, we prove our main theorems for families of abelian covers satsifying condition (*) . The first one asserts that if a 1-dimensional family of abelian covers satsifies condition (*) then it is a special family if it is of the form $P(\widetilde{G})$. The second one treats higher dimensional families that also have an eigenspace with $d_n=1$.  We first need a remark. 
\begin{remark} \label{irreduciblecover}
 In this section we will will work only with families of \emph{irreducible} abelian covers of
$\mathbb{P}^{1}$, i.e., the fibers of the family are irreducible curves. For cyclic covers this implies that the single row of 
the associated matrix is not annihilated by a non-zero element of $\mathbb{Z}/N\mathbb{Z}$. More generally, for abelian covers of $\mathbb{P}^{1}$ this implies
that the rows of the associated matrix are linearly independent over $\mathbb{Z}/N\mathbb{Z}$.
\end{remark}
\begin{theorem} \label{maintheorem1}
Let $f:\widetilde{\sC} \rightarrow T$ be a family of irreducible abelian covers which satisfies condition (*) with a fixed Prym datum $\Sigma\coloneqq(\tilde{G}, \tilde{\theta}_s,H)$ and with $s=4$, i.e., with 4 branch points. Then the associated subvariety $Z \subseteq A_{l,D}$ is a Shimura curve if and only if $Z=P(\widetilde{G})$.
 \end{theorem}
 \begin{proof}
One side is clear. To prove the other side, assume on the contrary that $Z \neq P(\widetilde{G})$ but $Z$ is a Shimura subvariety and we will derive a contradiction. Since $\dim Z=1$, the assumption $Z \neq P(\widetilde{G})$ implies that $\dim P(\widetilde{G})>1$. Then there are $a, a^{\prime} \in G $ with $a^{\prime}\neq \pm a$ such that $d_{a}=d_{-a}=1$ and $d_{a^{\prime}}=d_{-a^{\prime}}=1$. Consequently, using Lemma~\ref{dimeigspace}, the spaces $H^{1,0}_{-,n}$ for $n\in \{\pm a, \pm a^{\prime}\}$ are $1$-dimensional. For $n\in \{\pm a, \pm a^{\prime}\}$, let $\omega_n$ be the generator of $H^{1,0}_{-,n}$ as in (~\ref{basis}). For such $n$, the Hasse-witt matrix $A_{n}$ is a polynomial in $\mathbb{F}_{p}[z_{1},\dots,z_{4}]$. By the argument after Lemma~\ref{DWvanish}, there exist a prime number $p$ and an open subset $U$ of $T\otimes\mathbb{F}_{p}$ such that all fibers over $U$ are ordinary. Therefore the Hasse-Witt operator is an isomorphism over $U$ and so $A_{n}$ is invertible as a section of $\mathcal{O}_{U}$. It follows from the description of $\omega_n$ in ~\ref{basis} that $\omega_{a}.\omega_{-a}=\omega_{a^{\prime}}.\omega_{-a^{\prime}}$ as a section of the bundle $f_{*}(\omega^{\otimes 2})$. It is a non-zero section of the bundle $f_{*}(\omega^{\otimes 2})$ and so by Proposition~\ref{exactsequencevanish} we must have : \
 \[A_{a}.A_{-a}=A_{a^{\prime}}.A_{-a^{\prime}}\]
 as polynomials.\\

 We will show that this identity can not hold with the above
 conditions. The polynomials $A_{n}$ are given by the dual version
 of Lemma~\ref{HasseWitt} and we set $B_{n}=A_{n}\mid_{z_{1}=0}$. It means that we have : 
 \[B_{n}= \sum_{j_{2}+j_{3}+j_{4}=p-1}
 \binom{q.[-\alpha_{2}]_{N}}{j_{2}}...\binom{q.[-\alpha_{4}]_{N}}{j_{4}}z_{2}^{j_{2}}...z_{4}^{j_{4}}.\]
 For $h\in \{2,3,4\}$, let $r_{a}(h)$ be the largest integer $r$
 such that $B_{a}$ is divisible by $z^{r}_{h}$. We have that
 \[r_{a}(h)=\max \{0, q.\alpha_{k}+q.\alpha_{t}-(p-1)\},\]
where $\{k,t\}=\{2,3,4\}\setminus \{h\}$. We define $r_{-a}(h)$
similarly.\\

 Similarly let $r_{\pm a}(h)$ be the largest integer $\nu$ such
 that $B_{a}.B_{-a}$ is divisible by $z^{\nu}_{h}$. We have :
 \[r_{\pm a}(h)= q.\max\{\alpha_{1}+\alpha_{h},
 \alpha_{k}+\alpha_{t}\}-(p-1).\]
 Now the equality $A_{a}.A_{-a}=A_{a^{\prime}}.A_{-a^{\prime}}$
 implies that $r_{\pm a}(h)=r_{\pm a^{\prime}}(h)$ and so we get
 the following equality.
 \[\{\alpha_{1}+\alpha_{h},
 \alpha_{k}+\alpha_{t}\}=\{\alpha^{\prime}_{1}+\alpha^{\prime}_{h},
 \alpha^{\prime}_{k}+\alpha^{\prime}_{t} \}.\]
 One checks easily that this implies that there
 exists an even permutation $\sigma \in A_{4}$ of order $2$, such
 that $\alpha_{i}=\alpha^{\prime}_{\sigma(i)}$. It follows from Remark~\ref{irreduciblecover} that $\sigma \neq 1$. Therefore we again suppose, without loss of generality, that
 \[\alpha^{\prime}_{1}=\alpha_{2}, \alpha^{\prime}_{2}=\alpha_{1}\]
 \[\alpha^{\prime}_{3}=\alpha_{4}, \alpha^{\prime}_{4}=\alpha_{3}\]
 by our assumptions on $a$ and $a^{\prime}$, we have that
 \[\sum [\alpha_{i}]_{N}=\sum [\alpha^{\prime}_{i}]_{N}= 2N\]
 Assume first that $[\alpha_{1}]_{N}+[\alpha_{2}]_{N}=[\alpha_{3}]_{N}+[\alpha_{4}]_{N}=N$, or in other words, $[\alpha_{2}]_{N}=-[\alpha_{1}]_{N}$ and
 $[\alpha_{4}]_{N}=-[\alpha_{3}]_{N}$ in $\mathbb{Z}/N\mathbb{Z}$. This shows
 that the two rows $\alpha=(\alpha_{1},\dots,\alpha_{4})$ and
 $\alpha^{\prime}=(\alpha^{\prime}_{1},\dots,\alpha^{\prime}_{4})$ are linearly
 dependent which contradicts the irreducibility by Remark~\ref{irreduciblecover}.
 So the above equality does not hold and we may assume that
 $\alpha_{1}+\alpha_{2}<N$ and $\alpha_{3}+\alpha_{4}>N$. Now consider the row vector
 \[\alpha+\alpha^{\prime}=(\alpha_{1},\dots,\alpha_{4})+(\alpha^{\prime}_{1},\dots,\alpha^{\prime}_{4})=(\alpha_{1}+\alpha_{2},\alpha_{1}+\alpha_{2},\alpha_{3}+\alpha_{4},\alpha_{3}+\alpha_{4}).\]
 Note that the irreducibility assumption in Remark~\ref{irreduciblecover} assures that $\alpha+\alpha^{\prime} \neq \pm \alpha$ and one can easily verify that this row vector also satisfies the conditions for $\alpha_{i}$ and $\alpha^{\prime}_{i}$ and so we may replace the second row
 $(\alpha^{\prime}_{1},\dots,\alpha^{\prime}_{4})=(\alpha_{2},\alpha_{1},\alpha_{4},\alpha_{3})$ by
 this row vector and the equality
 $A_{a}.A_{-a}=A_{a^{\prime}}.A_{-a^{\prime}}$ should hold for this
 row vector as $\alpha^{\prime}$ and $(\alpha_{1},\dots,\alpha_{4})$ as $\alpha$ (actually for the corresponding elements $a,a^{\prime}\in G$). If this equality holds, the left hand side must contain a monomial of the form $z_{2}^{\alpha}z_{3}^{\beta}$ and also a monomial of the form $z_{1}^{\gamma}z_{4}^{\delta}$. This means that $\alpha_{2}+\alpha_{3}=\alpha_{1}+\alpha_{4}=\alpha_{1}+\alpha_{3}=\alpha_{2}+\alpha_{4}=N$ which is
 exactly to say that $\alpha=\alpha^{\prime}=(\alpha_{1},\alpha_{1},-\alpha_{1},-\alpha_{1})$. This contradiction completes the proof.
 \end{proof}

We also remark that if $\dim P(\widetilde{G})=s-3$, then, as explained earlier, the family is a special family, hence the condition $\dim P(\widetilde{G})>s-3$ is actually required
in the next theorem.
 
\begin{theorem} \label{mainthm1}
Let $\widetilde{\sC}\to T$ be a family of abelian covers of the line such that $\dim P(\widetilde{G})>s-3$ and the condition is satisfied. If there exists an eigenspace of type $(1,s-3)$, then the family does not give rise to a special subvariety of $A_{l,D}$. 
\end{theorem}
\begin{proof}
Assume on the contrary that $\widetilde{\sC}/T$ is a special family and take $n\in\widetilde{G}$ as in the assumption. We claim that there exists another $n^{\prime}\in\widetilde{G}$ such that $\{d_{n^{\prime}},d_{-n^{\prime}}\}=\{1,s-3\}$. Suppose that this is not the case. Observe that if for every $n\neq n^{\prime}\in\widetilde{G}$,
$d_{n^{\prime}}=0$ or $d_{-n^{\prime}}=0$, then by Lemma~\ref{dim P(G)}, 
$\dim P(\widetilde{G})=d_nd_{-n}=s-3$ which is against our assumptions. Hence
there exists $n^{\prime}\in\widetilde{G}$ such that $d_{n^{\prime}}\neq 0$
and $d_{-n^{\prime}}\neq 0$. If
$\{d_{n^{\prime}},d_{-n^{\prime}}\}\neq \{1,s-3\}$, then we have
an eigenspace of new type and consequently, $\dim P_f\geq
d_nd_{-n}+d_{n^{\prime}}d_{-n^{\prime}}>s-3$. So we may and do
assume that $\{d_{n^{\prime}},d_{-n^{\prime}}\}= \{1,s-3\}$. In this case both eigenspaces correspond to the same factor in the decomposition of $M^{ad}_{\mathbb{R}}$, see Remark~\ref{smallestshimura}. 
Let $p$ and $U$ be as in \S 2.1. So $p$ is a prime number such that $p\equiv 1 \text{ mod } N$ and set $q=\frac{p-1}{N}$. For these choices, and using Lemma~ \ref{dimeigspace}, consider the Hasse-Witt map
$\gamma_{(n)}:F_U^*\mathbb{E}_{-,(n)}\to \mathbb{E}_{-,(n)}$ and
let $\Gamma\in GL_{s-3}(\mathcal{O}_U)$ with respect to the basis
$\omega_{n,\nu}$ introduced earlier in ~\ref{basis}. Lemma~\ref{HasseWitt} gives a
description of the matrix $A=A_n$. We also denote
by $\gamma_{(n^{\prime})}$ and $A^{\prime}$ the corresponding
Hasse-Witt map and matrix respectively for $n^{\prime}$. We may, without loss of
generality, assume that $d_{-n}=d_{-n^{\prime}}=1$ and
$d_{n}=d_{n^{\prime}}=s-3$. Hence $A_{-n}$ and $A_{-n^{\prime}}$
are $1\times 1$ matrices, i.e., can be considered as sections of
$\mathcal{O}_U^*$ which we denote by $a, a^{\prime}$ respectively.
For each $\tau\in \{0,\dots, s-4\}$, let $\varphi_{\tau}\in
\Gamma(U,\mathcal{K}^G_-)$ be defined by
\[\varphi_{\tau}:=\omega_{-n,0}\otimes \omega_{n,\tau}-\omega_{-n^{\prime},0}\otimes \omega_{n^{\prime},\tau}\]
Note that since
$\omega_{-n,0}.\omega_{n,\nu}=\omega_{-n^{\prime},0}.
\omega_{n^{\prime},\nu}$ as sections of $f_{*}(\omega^{\otimes
2})$, it follows that the image of $\varphi_{\tau}$
under $\Sym^2(\gamma)$ is equal to
\[\displaystyle \sum_{\nu=0}^{s-4}(a.\Gamma_{\nu, \tau}-a^{\prime}\Gamma_{\nu, \tau})(\omega_{-n,0}.\omega_{n,\nu}).\]
But the sections $\omega_{-n,0}.\omega_{n,\nu}$ are linearly
independent for $\nu\in \{0,\dots, s-4\}$, so Proposition~\ref{exactsequencevanish}
gives that $a\Gamma_{\nu, \tau}-a^{\prime}\Gamma_{\nu, \tau}=0$
for all $\tau, \nu \in \{0,\dots, s-4\}$. This can be rewritten as 
$a^{-1}A_{\nu, \tau}=a^{\prime -1}A^{\prime}_{\nu, \tau}$. By examining the powers of various indeterminates in both sides of this equation, 
we show that this equality cannot hold. Assume that the equality holds. Choose two distinct $i,j\in \{1,\dots, s\}$ and set $I=\{1,\dots, s\}\setminus \{i,j\}$. For $h= 1,2$, define $r_n(h)$ similarly as in the proof of Theorem~\ref{maintheorem1}, to be the largest integer $r$ such that $A_{h,h}|_{t_i=0}$ is divisible by $t^{r}_{j}$. Similarly, let $r_{-n}$ be the largest integer $r$ such that $a^{-1}|_{t_i=0}$ is divisible by $t^{r}_{j}$. Let $(\tilde\alpha_1,\dots, \tilde\alpha_s)$ be as in Lemma~\ref{eigenbundleformula}. By the formulas for $a^{-1}$ and the matrix $A$ given in Lemma~\ref{HasseWitt}, we find
\[r_{-n}=\displaystyle \max \{0,(p-1).q\sum_{i\in I}[\tilde\alpha_i]_N\}\]
\[r_{n}(1)=\displaystyle \max \{0,(s-3)(p-1).q\sum_{i\in I}[-\tilde\alpha_i]_N\}.\]
Note that $r_{n}(2)=0$.
With the above definitions, $u_n(h)=r_{-n}+r_n(h)$ is the largest
integer $u$ such that $(a^{-1}A_{h,h})|_{t_i=0}$ is divisible by
$t_{j}^u$. The fact that $d_{-n}=1$ and $d_{n}=s-3$ show that
$\sum [\tilde\alpha_i]_N=2N$ and $[-\tilde\alpha_i]_N=N-[\tilde\alpha_i]_N$ for
every $i$. By these relations one gets
\[u_n(1)=\displaystyle \max \{(p-1)-q\sum_{i\notin I} [\tilde\alpha_i]_N, (p-1)-q\sum_{i\in I} [\tilde\alpha_i]_N\}
=|(p-1)-q\sum_{i\notin I}
[\alpha_i]_N|=q.|N-[\tilde\alpha_k]+[\tilde\alpha_{\lambda}]|,\]
\[u_n(2)=\displaystyle \max \{0, (p-1)-q\sum_{i\notin I} [\tilde\alpha_i]_N\}=\max \{0, (p-1)-q\sum_{i\in I} [\tilde\alpha_i]_N\}=
q.\max \{0, [\tilde\alpha_k]+[\tilde\alpha_{\lambda}]-N\}\]

Analogously, we can define the above notions for $\pm n^{\prime}$
which we represent by $u^{\prime}(1), u^{\prime}(2)$. The above
relations give that $u(1)=u^{\prime}(1)$, $u(2)=u^{\prime}(2)$.
This implies that for all $k,\lambda$, we have
$[\tilde\alpha_k]_N+[\tilde\alpha_{\lambda}]_N=[\tilde\alpha^{\prime}_k]_N+[\tilde\alpha^{\prime}_{\lambda}]_N$.
From this one gets that $[\tilde\alpha_i]_N=[\tilde\alpha^{\prime}_i]_N$ for every $i\in \{1,\cdots, s\}$. But
this implies that the two rows corresponding to $n$ and
$n^{\prime}$ are equal and in particular linearly dependent. This
is in contradiction with our assumptions by Remark~\ref{irreduciblecover}.
\end{proof}


\begin{thebibliography} {99}
\bibitem{Andr}
Y. Andr\'e, \emph{Mumford-Tate groups of mixed Hodge structures and the theorem
of the fixed part.} Compositio Math. 82 (1992), 1–24.

\bibitem{BL}
	C. Birkenhake, H. Lange, \emph{Complex abelian varieties.} Volume 302 of Grundlehren der Mathematischen Wissenschaften [Fundamental Principles of Mathematical Sciences]. Springer-Verlag, Berlin, second edition, 2004.

\bibitem{CF2}
	E. Colombo, P. Frediani, \emph{On the dimension of totally geodesic submanifolds in the Prym loci.} arXiv:2101.05189. To appear in Boll Unione Mat Ital (2021). https://doi.org/10.1007/s40574-021-00287-4.

\bibitem{CFGP}
		E. Colombo, P. Frediani, A. Ghigi, M. Penegini, \emph{Shimura curves in the Prym locus.}  Communications in Contemporary Mathematics, Vol. 21, N0. 2 (2019) 1850009 (34 pages).
\bibitem{DM}
P. Deligne ,  G. Mostow, \emph{Monodromy of hypergeometric functions and non-lattice integral monodromy.} Inst. Hautes \'Etudes Sci. Publ. Math. 63 (1986), 5-89.

\bibitem{DO}
B. Dwork, A. Ogus, \emph{Canonical liftings of jacobians}. Compositio Math. 58 (1986), 111-131.

\bibitem{F}
P. Frediani, \emph{Abelian covers and second fundamental form.}  arXiv:2105.07947. 

\bibitem{FG}
	G.P. Grosselli, P. Frediani, \emph{Shimura curves in the Prym loci of ramified double covers.}  arXiv:2007.09646. 
	
\bibitem{FGP}
P. Frediani, A. Ghigi, M. Penegini, \emph{Shimura varieties in the Torelli locus via
Galois coverings.} Int. Math. Res. Not. IMRN, (20):10595-10623, 2015.

\bibitem{H}
S. Helgason, \emph{Differential geometry, Lie groups, and
symmetric spaces}. Pure and Applied Mathematics 80, Academic
Press, Inc., New York-London, 1978.

\bibitem{M21}
A. Mohajer, \emph{On Shimura subvarieties of the Prym locus.} Comm. Algebra 49 (2021), no. 6, 2589-2596.

\bibitem{MZ}
A. Mohajer, K. Zuo, \emph{On Shimura subvarieties generated by families of abelian covers of $\mathbb{P}^{1}$.} Journal of Pure and Applied Algebra, 222 (4), 2018, 931-949.

\bibitem{M10}
B. Moonen, \emph{Special subavarieties arising from families of
cyclic covers of the projective line}. Documenta Mathematica 15
(2010) 793-819.

\bibitem{Noo}
R. Noot, \emph{Models of Shimura varieties in mixed characteristic.} 
J. Algebraic Geom. ,5(1):187-207, 1996.

\bibitem{no} 
Naranjo, ~J.C., Ortega, ~ A., Verra, ~A.,  \emph{Global Prym-Torelli for double coverings ramified in at least 6 points.} arXiv:2005.11108. To appear in Journal of Algebraic Geometry.

\bibitem{R}
J. Rohde, \emph{Cyclic coverings, Calabi-Yau manifolds and complex multiplication}. Lecture Notes in Math. 1975, Springer, 2009.

\bibitem{RR}
S. Recillas, R. Rodr\'iguez, \emph{Prym varieties and fourfold covers.} Publ. Preliminares  Inst.Mat. Univ. Nac. Aut. Mexico, 686(2001), arXiv:math/0303155, 2003.
\end{thebibliography}
\end{document}